\documentclass[10pt,leqno]{amsart}
\usepackage{amscd}
\usepackage{color}
\usepackage{amssymb}
\usepackage{amsfonts}
\usepackage{latexsym}
\usepackage{verbatim}
\usepackage{amsthm}
\numberwithin{equation}{section}

\theoremstyle{plain}
\newtheorem{theorem}{Theorem}[section]
\newtheorem{corollary}[theorem]{Corollary}

\newtheorem{proposition}[theorem]{Proposition}
\newtheorem{lemma}[theorem]{Lemma}

\theoremstyle{definition}
\newtheorem{definition}[theorem]{Definition}
\newtheorem{example}[theorem]{Example}
\newtheorem{oss}[theorem]{Remark}

\newcommand{\de}{\partial}

\newcommand\C{{\mathbb C}}
\newcommand\R{{\mathbb R}}

\newcommand{\ov}[1]{\overline{ #1}}

\newcommand{\N}{\nabla}
\renewcommand{\bar}[1]{\overline{#1}}

\newcommand{\bk}{\overline{k}}

\newcommand{\br}{\overline{r}}

\newcommand{\g}{\mathfrak{g}}

\begin{document}
\title[The pluriclosed flow on nilmanifolds]{The pluriclosed flow on nilmanifolds and Tamed symplectic forms}
\author{Nicola Enrietti, Anna Fino  and Luigi Vezzoni}
\date{\today}
\address{Dipartimento di Matematica G. Peano \\ Universit\`a di Torino\\
Via Carlo Alberto 10\\
10123 Torino\\ Italy} \email{nicola.enrietti@unito.it}
 \email{annamaria.fino@unito.it}
 \email{luigi.vezzoni@unito.it}

\thanks{Research partially supported by GNSAGA (Indam) of Italy.}
\maketitle
\begin{abstract}
We study evolution of (strong K\"ahler with torsion) SKT structures via the pluriclosed flow on complex  nilmanifolds, i.e.  on compact quotients of simply connected nilpotent Lie groups  by discrete subgroups endowed with an invariant complex structure.
Adapting   to our case the techniques introduced by Jorge Lauret for studying Ricci flow on homogeneous spaces    we show  that   for SKT  Lie algebras  the pluriclosed flow  is equivalent to a bracket flow and   we prove a long time existence result in the nilpotent case. Finally, we introduce a natural flow for evolving tamed
symplectic forms on a complex manifold, by  considering  evolution of symplectic forms via the flow induced by the Bismut Ricci  form.

\end{abstract}
\section{introduction}
Let $M$ be a   Hermitian manifold of complex dimension $n$.   If $M$ is  non-K\"ahler, the Levi-Civita connection is not compatible with the induced ${\rm U}(n)$-structure and its role is often replaced by other connections having torsion but preserving the Hermitian structure \cite{Gauduchon}. Although
there is no a canonical choice of a Hermitian connection, the Chern and the Bismut connections seem to have a central role. The
 Chern connection is defined as the unique Hermitian connection $\nabla^C$ for which  the $(1,1)$-component of the torsion tensor vanishes, while    the Bismut connection
 has skew-symmetric torsion \cite{Bismuth}. Streets and Tian pointed out in \cite{ST2} that the operator $g\mapsto -S(g)$ defined from the space of Hermitian metrics on a complex manifold $(M,J)$ by the formula
$$
S(g)_{i\bar j}=g^{\bar k r} R^{C}_{i\bar j r\bar k}
$$
is strongly elliptic, where $R^{C}=[\nabla^C,\nabla^C]-\nabla^C_{[\,,\,]}$ is the curvature of $\nabla^C$.  Consequently the standard theory of parabolic equations ensures that the Ricci-type flow
\begin{equation}\label{preflow}
\begin{cases}
\frac{d}{dt}g=-S(g)+L(g)\\
g(0)=g_0
\end{cases}
\end{equation}
has a unique maximal solution defined in an interval $[0,T)$, where $L$ is an arbitrary first order differential operator.  Moreover,  if we choose $L(g)$ to be a suitable operator $Q(g)$ depending quadratically on the torsion of $\N^C$,  the flow
\begin{equation}\label{HCF}
\begin{cases}
\frac{d}{dt}g=-S(g)+Q(g)\\
g(0)=g_0
\end{cases}
\end{equation}
is  the gradient flow of a  functional $\mathbb{F}=\mathbb{F}(g)$. The flow \eqref{HCF} is called
the {\em Hermitian curvature flow} and preserves both the K\"ahler and the SKT condition (see \cite{ST}). We recall that a Hermitian structure  $(J, g)$  is called SKT (strong K\"ahler with torsion)  or {\em{pluriclosed}} if its fundamental form  $\omega(\cdot,\cdot)=g(\cdot,J\cdot)$  is $\partial\ov{\partial}$-closed.
In the SKT case, \eqref{HCF}  is equivalent to the  so-called {\em pluriclosed  flow}
\begin{equation}\label{flowST}
\begin{cases}
\frac{d}{dt}\omega=-(\rho^{B})^{1,1}\\
\omega(0)=\omega_0
\end{cases}
\end{equation}
acting on the space of $J$-compatible non-degenerate real $2$-forms, where $(\rho^B)^{1,1}$ denotes the $(1,1)$-part of   the Ricci form $\rho^B$  of the Bismut  connection (i.e.  the  $(1,1)$-part of  the so-called   Bismut Ricci form). In the terminology of \cite{ST} an SKT structure $\omega$ is called {\em static} if it satisfies the Einstein-type equation
\begin{equation}\label{static}
r\,\omega=(\rho^{B})^{1,1}
\end{equation}
where $r\in \R$. Static SKT structures seem to be very rare in complex non-K\"ahler manifolds, since if $\omega$ is static with $r\neq 0$, then $\Omega=\frac{1}{r}\rho^B$ is a symplectic form taming $J$ (i.e. a {\em Hermitian-symplectic structure}  in the terminology of \cite{ST}).   If a complex surface admits a   Hermitian-symplectic structure $\Omega$, then by \cite{Popovici} the Hermitian metric associated to $\Omega^{1,1}$ is {\em strongly Gauduchon}. Indeed, by \cite[Lemma 3.2]{Popovici}  a complex manifold $(M, J)$ of complex dimension $n$  carries a strongly
Gauduchon metric  $g$  if and only if there exists a real $d$-closed ${\mathcal C}^{\infty}$ $(2n-2)$-form
$\Omega$ on $M$ such that its component of type $(n-1, n-1)$
 is positive on $M$.

It is known that every compact complex surface admitting a Hermitian-symplectic structure is actually K\"ahler (see \cite{ST,LZ}) and it is still an open problem to find an example of a compact Hermitian-symplectic manifold  not admitting K\"ahler structures. Some negative results on this context are proved in \cite{EFV} for nilmanifolds and in \cite{EF} for $4$-dimensional Lie algebras. In particular, from the results of \cite{EFV} it turns out that a nilmanifold  endowed with an invariant complex structure  does not admit Hermitian-symplectic structures  unless it is a torus. This last result together  with a theorem of \cite{Enrietti} implies that   complex nilmanifolds    cannot admit  static SKT structures unless they are tori.

\medskip
The present paper is divided in two parts. In the first one we investigate  the behaviour of solutions of \eqref{flowST} on Lie algebras. In particular we prove the following
\begin{theorem}\label{main1}
Let $(M=G/ \Gamma,J,\omega_0)$ be a {\em nilmanifold}  endowed with an invariant SKT structure. Then the solution $\omega(t)$ to the pluriclosed flow \eqref{flowST} is defined for every $t\in (-\epsilon,\infty)$, where $\epsilon$ is a suitable positive real  number.
\end{theorem}
A key ingredient in the proof of this theorem is a trick introduced by Lauret in \cite{lauret2} for studying the Ricci flow on homogeneus spaces evolving the Lie brackets instead of  the Riemannian metrics.

\medskip
In the second part of the paper we introduce a natural flow for evolving  taming symplectic forms on a complex manifold $(M, J)$.
Given  a symplectic form  $\Omega_0$  taming  $J$,   we consider the flow
\begin{equation}\label{eq:HSflow}
 \begin{cases}
  \frac{d}{d t}\Omega=-\rho^B ( \omega),\\
  \Omega(0)=\Omega_0
 \end{cases}
\end{equation}
where $\rho^B (\omega)$ is the Bismut Ricci form of the Hermitian metric associated to $\omega = \Omega^{1,1}$.  For such a flow we prove a short time existence result and a stability theorem involving K\"ahler-Einstein metrics.

\medskip
\noindent {\em{Acknowledgements}}. The authors would like to thank  Jorge Lauret  for  useful  conversations and the referee for helpful comments on the paper.

\section{Preliminaries on SKT metrics}
Let $(M,g,J)$ be a Hermitian manifold with fundamental form $\omega$. The form $\omega$ and the Riemannian metric $g$ are related 
$$
\omega(\cdot,\cdot)=g(\cdot,J\cdot)\,.
$$
We denote by $\N^B$ the Bismut connection of $(g,J)$. This  connection was introduced by   Bismut in \cite{Bismuth}
and it is the unique Hermitian connection (i.e. $\N^BJ=0$, $\N^B g=0$) such that
\begin{equation} \label{definzc}
c(X,Y,Z):=g(X,T^B(Y,Z))
\end{equation}
is  a $3$-form,  where 
$$
T^{B}(X,Y)=\N^B_{X}Y-\N^B_YX-[X,Y]
$$
denotes  the torsion of $\N^B$.  This connection induces the curvature tensor
$$
R^B(X,Y):=[\N^B_X,\N^B_Y]-\N^B_{[X,Y]},
$$
 the Ricci tensor and the Ricci form given respectively by
$$
ric^B(X,Y)=g^{kr}R^B(e_k,X,Y,e_r)\,,\quad  \rho^B(X,Y)=\frac12g^{kr}g(R^B(X,Y)e_k,Je_r)\,,
$$
with  $\{e_k\}$  an arbitrary  local frame. In complex notation we can alternatively write
$$
ric^B(X,Y)=-ig^{\bar r k}R^B(Z_k,X,Y,Z_{\bar r})\,,\quad\rho^B(X,Y)=-ig^{\bar r k}g(R^B(X,Y)Z_k,Z_{\bar r})\,,
$$
where $\{Z_k=\frac12 (e_k-iJe_k)\}$ is a   local Ê$(1,0)$-frame and $Z_{\bar k}=\frac12 (e_k+iJe_k)$.

For a Hermitian manifold $(M, J)$ the Ricci tensor of $\N^B$ and the usual Ricci tensor are related by the following formula \begin{equation}\label{fundamental}
ric^B (X, Y) = ric^g (X, Y) - \frac 12   (d^* c) (X,Y) - \frac 14 \sum_{i = 1}^{2n} g (T^B (X, e_i), T^B (Y, e_i))\,,
\end{equation}
(see \cite{ivanov2}), where $d^*$ is the co-differential operator of $g$, while $\rho^B$ is related to the Ricci form  $\rho^C$ of the Chern connection by
\begin{equation}\label{fundamental2}
\rho^B=\rho^C-dd^*\omega\,.
\end{equation}

\medskip
We recall the  following
\begin{definition}
A Hermitian metric  $g$ on a complex   manifold $(M, J)$ is called SKT (strong K\"ahler with torsion)  or pluriclosed if   the torsion 3-form $c$ is closed or, equivalently, if  its associated fundamental form
$\omega$ satisfies  $\de \bar{\de}\omega=0$.
\end{definition}

For a complex surface, an  SKT metric  is  called {\em standard}
in the terminology of Gauduchon  \cite{gauduchon2}. In the conformal
class of any given Hermitian metric  on a compact  complex manifold there always exists a standard metric. But  this property is not anymore true in higher dimensions for SKT metrics.

For real Lie algebras admitting left-invariant SKT metrics there are some classification results in dimensions  four, six and  eight.  More precisely, 6-dimensional (resp. 8-dimensional) SKT nilpotent Lie algebras have been classified in \cite{FPS} (resp. in \cite{EFV} and for a particular class in \cite{RT}) and a classification of  $4$-dimensional SKT solvable Lie algebras has been obtained in \cite{MS}.

General results are known  for nilmanifolds, i.e.  for compact quotients of simply connected nilpotent Lie groups $G$ by discrete subgroups  $\Gamma$. Indeed, in [11] it has been shown that if $(M=G/\Gamma,J)$  is a nilmanifold (not a torus) endowed with an invariant complex structure $J$ and  an  SKT metric $g$ compatible with $J$, then  the nilpotent  Lie group $G$ must be $2$-step nilpotent and $M$ is a total space of a principal holomorphic torus bundle over a torus.

\section{Pluriclosed Flow  on Lie algebras}
Let $G$ be a Lie group with a left-invariant  SKT  structure $(g_0,J)$ and let $\Gamma$ be a co-compact lattice in $G$. We are interested in studying solutions to \eqref{flowST}  on  the compact manifold $M=G/\Gamma$ endowed with  the invariant SKT structure induced by $(g_0,J)$. Since  the  pluriclosed flow \eqref{flowST} is invariant by biholomorphisms of  the complex manifold $(M,J)$, when $\omega_0$ is invariant,   the solution $\omega(t)$ to \eqref{flowST} is invariant for every $t$. Therefore the PDE system \eqref{flowST}  on   $M=G/\Gamma$ is equivalent to an ODE system on the Lie algebra $\g$ of $G$. In general neither expect that the flow converges nor that its limit is non-degenerate. 

\medskip
Let $(\mathfrak{g},\mu)$ be a Lie algebra endowed  with an  SKT  structure $(g,J)$,  where $\mu$ denotes the  {Lie}  bracket on $\mathfrak g$. By an  SKT structure  on a Lie algebra we means a pair $(g,J)$, where  $J$ is a complex structure, satisfying the
 integrability condition
$$
\mu(JX,JY)=J\mu(JX,Y)+J\mu(X,JY)+\mu(X,Y),
$$
for every $X,Y\in \g$ and  $g$  is an inner product  compatible with $J$ and such that  $dc =0$, where $c$ is defined by \eqref{definzc}.
In order to write down a formula for the Bismut Ricci form  $\rho^B$ in this  algebraic context, we fix an arbitrary $(1,0)$ frame $\{Z_r\}$  of $(\g,J)$ with dual frame $\{\zeta^k\}$. Following the approach of \cite{Vezzoni} we can write\\
$$
\rho^B=d\eta\,,
$$
where $\eta$ is the real $1$-form

\begin{equation}\label{theta1}
\eta (X)=\Im\mathfrak{m}\left\{g^{\bk r}g(\mu(X -i JX,Z_r),Z_{\bk})\right\}+i g^{\bk r}g(\mu(Z_r,Z_{\bk}),X).
\end{equation}
In complex notation we  have
$$
\eta=\eta_a\,\zeta^a+\eta_{\bar b}\, \zeta^{\bar b},
$$
where
\begin{equation}\label{theta2}
\eta_a=-i g^{\bk r}g(\mu(Z_a,Z_r),Z_{\bk})+ig^{\bk r}g(\mu(Z_r,Z_{\bk}),Z_a)\,,\quad \eta_{\bar b}=\bar{\eta_{b}}\,.
\end{equation}
Formula \eqref{theta2} can be rewritten in terms of the  Lie bracket components $\mu_{ij}^k$ as
$$
\eta_a=-i \mu_{ar}^{r}+ig^{\bk r}\mu_{r\bk}^{\bar l} g_{a\bar l}\,,\quad  \eta_{\bar b}=\bar{\eta_{b}}\,.
$$
Therefore, in complex notation   we obtain
$$
\rho^B=-i\,\rho^B_{i\bar j}\,\zeta^{i}\wedge \zeta^{\bar j}-\frac{i}{2}\,\rho^B_{hk}\,\zeta^{h}\wedge \zeta^{k}-\frac{i}{2}\,\rho^B_{\bar l\bar m}\,\zeta^{\bar l}\wedge \zeta^{\bar m}\,,\quad \rho^B_{i\bar j}=g^{\bar s k}R^B_{i\bar js\bar k},
$$
where
$$
\begin{aligned}
\rho^B_{i\bar j}=&\,-i\eta(\mu(Z_i,Z_{\bar j}))=-i\mu_{i\bar j }^a\eta_a-i\mu_{i\bar j }^{\bar b}\eta_{\bar b}\\
                =&\,- \mu_{i\bar j }^a\mu_{ar}^{r}+ \mu_{i\bar j }^ag^{\bk r}\mu_{r\bk}^{\bar l} g_{a\bar l}
                 +\mu_{i\bar j }^{\bar b}\mu_{\bar b \bar r}^{\bar r}+\mu_{i\bar j }^{\bar b}g^{k \br}\mu_{k \bar r }^{ l} g_{l\bar b}\,,
\end{aligned}
$$
i.e.,
\begin{equation}\label{rho1}
\rho^B_{i\bar j}=\,- \mu_{i\bar j }^a\mu_{ar}^{r}+ \mu_{i\bar j }^ag^{\bk r}\mu_{r\bk}^{\bar l} g_{a\bar l}
                 +\mu_{i\bar j }^{\bar b}\mu_{\bar b \bar r}^{\bar r}+\mu_{i\bar j }^{\bar b}g^{k \br}\mu_{k \bar r }^{ l} g_{l\bar b}\,.
\end{equation}
In the same way
\begin{equation}\label{rho2}
\rho^B_{i j}=\, \,- \mu_{i j }^a\mu_{ar}^{r}+ \mu_{i j }^ag^{\bk r}\mu_{r\bk}^{\bar l} g_{a\bar l}
                 +\mu_{i j }^{\bar b}\mu_{\bar b \bar r}^{\bar r}+\mu_{i j }^{\bar b}g^{k \br}\mu_{k \bar r }^{ l} g_{l\bar b}
\end{equation}
and  \eqref{HCF} writes as
\begin{equation}\label{flowSTalgebra}
\begin{cases}
\frac{d}{dt}g_{i\bar j}=\,\,- \mu_{i\bar j }^a\mu_{ar}^{r}+ \mu_{i\bar j }^ag^{\bk r}\mu_{r\bk}^{\bar l} g_{a\bar l}
                 +\mu_{i\bar j }^{\bar b}\mu_{\bar b \bar r}^{\bar r}+\mu_{i\bar j }^{\bar b}g^{k \br}\mu_{k \bar r }^{ l} g_{l\bar b}\\
                 g_{i\bar j}(0)=(g_0)_{i\bar j}\,.
\end{cases}
\end{equation}

Note that when  $\{Z_r\}$ is a unitary frame we have
\begin{eqnarray}\label{unitary}
&&\rho^B_{i\bar j}=\,\,- \mu_{i\bar j }^a\mu_{ar}^{r}+ \mu_{i\bar j }^a\mu_{r\bar r}^{\bar a}
                 +\mu_{i\bar j }^{\bar b}\mu_{\bar b \bar r}^{\bar r}+\mu_{i\bar j }^{\bar l}\mu_{r \bar r }^{ l}\\
&&\rho^B_{i j}=\,- \mu_{ij }^a \mu_{ar}^{r}+\mu_{ij }^l\mu_{k\bk}^{\bar l}
                 +\mu_{i j }^{\bar b}\mu_{\bar b \bar r}^{\bar r}+\mu_{i j }^{\bar l}\mu_{k \bar k}^{ l} \,,
\end{eqnarray}
where the repeated indexes are summed.

If the Lie algebra $\mathfrak g$ is $2$-step nilpotent, i.e.  if $$
\mu(\mu(X,Y),Z)=0
$$
for every $X,Y,Z\in \g$, then  the formulas \eqref{rho1} and \eqref{rho2} reduce to
\begin{eqnarray}
&&\rho^B_{i\bar j}=\,\mu_{i\bar j }^ag^{r\bk}\mu_{r\bk}^{\bar l} g_{a\bar l}+\mu_{i\bar j }^{\bar b}g^{k \br}\mu_{k \bar r }^{ l} g_{l\bar b}\,.\\
&&\rho^B_{i j}=\,\mu_{ij }^ag^{r\bk}\mu_{r\bk}^{\bar l} g_{a\bar l}+\mu_{i j }^{\bar b}g^{k \br}\mu_{k \bar r }^{ l} g_{l\bar b}\,.
\end{eqnarray}
giving the suitable expression
\begin{equation}\label{rhonilpotent}
\rho^B(X,Y)=-ig^{\bar k r}\, g\left(\mu(X,Y),\mu(Z_r,Z_{\bar k})\right),
\end{equation}
for every $X,Y\in\g$.

\begin{oss}
We observe that the  previous computations hold also in the non-SKT case.
\end{oss}

\medskip

Next we show  two examples of   SKT Lie algebras in dimension $4$ for which the   solution $\omega(t)$  of  the pluriclosed flow \eqref{flowST} is defined for every $t\in (-\epsilon,\infty)$, where $\epsilon$ is a suitable positive  real  number.  The first example is nilpotent and we will show in the next  section that  this happens for every SKT nilpotent Lie algebra. The second one is solvable and  admits  a generalized K\"ahler structure \cite{AG,FT3}.

\begin{example}\label{KT}
In dimension $4$ the unique nilpotent SKT Lie algebra up to isomorphisms
is  $\mathfrak{h}_3 \oplus\R$, where $\mathfrak{h}_3$ is the  Lie algebra of the 3-dimensional real Heisenberg Lie group $H_3 (\R)$  given by
$$
 H_3 (\R) = \left \{  \left (    \begin{array}{ccc} 1&x&z\\ 0&1&y\\ 0&0&1  \end{array}  \right ), \quad x, y, z \in \R  \right \}.
$$
 The  compact quotient of the corresponding simply-connected $H_3 (\R) \times \R$ by the lattice $\Gamma\times \mathbb Z $, where $\Gamma$ is the lattice in $H_3 (\R)$ whose elements  are matrices with integer entries, is the so-called Kodaira-Thurston surface.
The Lie algebra  $\g=\mathfrak h_3\oplus\R$  has structure equations $(0,0,0,12)$, where by this notation we mean that there exists a basis of 1-forms $\{ e^i \}$ such that
$$
d e^i = 0, i = 1,2,3, \quad de^4 = e^1 \wedge e^2.
$$
 Let $J$ be the complex structure on $\g$  given by
$$
Je_1=-e_2,\,\quad  Je_3=-e_4\,.
$$
Then
$$
Z_1=\frac12\, (e_1+ie_2)\,,\quad Z_2=\frac12\,(e_3+ie_4)
$$
is a complex basis of type $(1,0)$ of $(\g,J)$. Let $\{\zeta^1,\zeta^2\}$ be  its dual frame.
Every Hermitian inner product  $g$ on $(\g,J)$ can be written as
$$
g=x\zeta^{1}\zeta^{\bar 1}+y\zeta^{2}\zeta^{\bar 2}+z\zeta^{1}\zeta^{\bar 2}+\bar z \zeta^{2}\zeta^{\bar 1}\,.
$$
where $x,y\in \R$, $z\in \C$ satisfy $xy-|z|^2>0$ and  it is SKT.
Since
$$
\mu(Z_1,Z_{\bar 1})=-\frac12(Z_2-Z_{\bar 2})
$$
is the only non-vanishing bracket we  have
$$
\rho^{B}=-i\rho^{B}_{1\bar 1}\, \zeta^{1\bar 1}
$$
where
$$
\rho^{B}_{1\bar 1}=-i\eta([Z_1,Z_{\bar 1}])=i\frac12 (\eta_2-\eta_{\bar 2})=-\,\Im\mathfrak{m}\,\eta_2\,.
$$
A direct computation yields
$$
\eta_{2}=i \frac{y^2}{2\left(xy-|z|^2\right)}
$$
and
$$
\rho^{B}_{1\bar 1}= -\frac{y^2}{2\left(xy-|z|^2\right)}\,.
$$
Therefore in this case system \eqref{flowSTalgebra} reduces to
\begin{equation}
\label{KT}
\begin{cases}
\dot{x}=\frac{y^2}{2\left(xy-|z|^2\right)}\\
y\equiv y_0\,,\quad z\equiv z_0\\
x(0)=x_0\,,
\end{cases}
\end{equation}
and  the solution to \eqref{flowRn} with
$$
\omega_0=-ix_0\zeta^{1\bar 1}-iy_0\zeta^{2\bar 2}-iz_0\zeta^{1\bar 2}-i\bar z_0\zeta^{2\bar 1}
$$
is
$$
\omega(t)=-ix(t)\zeta^{1\bar 1}-iy_0\zeta^{2\bar 2}-iz_0\zeta^{1\bar 2}- i \bar z_0\zeta^{2\bar 1}
$$
where
$$
x(t)=\frac{1}{y_0}\left(\sqrt{y_0^2 t+(x_0y_0-|z_0|^2)^2}+|z_0|^2\right)\,.
$$
For instance if we start from the standard SKT structure
$$
\omega_0=-i\zeta^{1\bar 1}-i\zeta^{2\bar 2}
$$
we get
$$
\omega(t)=-i\sqrt{1+t}\,\zeta^{1\bar 1}-i\zeta^{2\bar 2}\,.
$$
\end{example}

\begin{example}  Consider the solvable Lie algebra with structure equations
$$
\left \{ \begin{array}{l}
d e^1 = a e^{14} + b e^{24},\\[3 pt]
d e^2 = - b e^{14} + a e^{24},\\[3 pt]
d e^3 = - 2 a e^{34}\\[3 pt]
d e^4 =0,
\end{array}
\right . \quad  a, b \in \R - \{ 0\}, 
$$
endowed with the  complex structure given by
$$
Je_1 = e_2,\quad J e_3 = e_4.
$$
A compact quotient of the corresponding simply-connected Lie group by a  lattice is a   Inoue surface of type $S^0$ (see \cite{Inoue}).
Let $\{ Z_1, Z_2 \}$ be the $(1,0)$-frame
$$
Z_1 = \frac{1}{2} (e_1 - i e_2), \quad Z_2 = \frac{1}{2} (e_3 - i e_4)\,;
$$
then a direct computation yields
$$
\begin{array}{l}
\mu(Z_1,  Z _{\ov 1} ) =0, \quad \mu(Z_1, Z_2) = \lambda\, Z_1, \quad \mu(Z_{\ov 1}, Z_{\ov 2}) =\bar \lambda Z_{\ov 1}, \\[4 pt]
\mu(Z_{ 1}, Z_{\ov 2}) =-\lambda\, Z_{1}, \quad  \mu(Z_{ \ov1}, Z_{2}) = -\bar \lambda Z_{\ov1}, \quad \mu(Z_2, Z_{\ov 2}) = a i (Z_2+Z_{\bar 2})
\end{array}
$$
where
$$
\lambda=\frac{b+ia}{2}\,.
$$
Consider the $(1,0)$-coframe
$$
\zeta^1=e^1+ie^2\,,\quad \zeta^2=e^3+i e^4
$$
dual to $\{Z_1,Z_2\}$. Then
$$
d\zeta^1= - \lambda\,  (\zeta^{12}- \zeta^{1 \ov 2}), \quad d\zeta^2= -  ai \,  \zeta^{2 \ov 2}.
$$
Let $g$ be an arbitrary $J$-Hermitian metric on $\g$. We can write
$$
g=x\,\zeta^{1}\zeta^{\bar1}+ y \zeta^{2}\zeta^{\bar2} + z \zeta^{1}\zeta^{\bar2}+ \bar z \zeta^{2}\zeta^{\bar1}
$$
where $x,y\in \R_+$, $z\in \C$ satisfy
$$
xy-|z|^2>0\,.
$$
The fundamental form of $g$ is
$$
\omega=-ix\,\zeta^{1\bar 1}- iy \zeta^{2\bar 2} - iz \zeta^{1\bar2}- i\bar z \zeta^{2\bar1}\,.
$$
Let $\rho^B=d\eta$ be the Bismut form of $g$. A  direct computation yields
$$
\eta_1=-\frac{xz}{xy-|z|^2}\left( i \bar \lambda+a\right)=-\frac{3a+ib}{2}\,\frac{xz}{xy-|z|^2}
$$
and
$$
\eta_2=i\lambda-\frac{i\bar \lambda|z|^2 +axy}{xy-|z|^2}-\frac{i|z|^2(\lambda+\bar \lambda ) +xy(a-i\lambda)}{xy-|z|^2}\,.
$$

In matrix notation we have
$$
(g_{i\bar j})=\left(\begin{array}{cc}
x &z\\
\bar z & y\\
\end{array}
\right)
$$
and
$$
(g^{\bar ji})=\frac{1}{xy-|z|^2}\left(\begin{array}{cc}
y &- z\\
- \bar z & x\\
\end{array}
\right).
$$
A direct computation yields
\begin{eqnarray}
&& \rho^B_{1\bar 1}=0,\\
&& \rho^B_{1\bar 2}=\lambda\theta_1=\frac{1}{4}(3a^2+b^2-2iab)\,\frac{xz}{xy-|z|^2},\\
&& \rho^B_{2\bar 2}=-\frac{3a^2xy}{xy-|z|^2}
\end{eqnarray}
and the ODEs induced by \eqref{flowSTalgebra} are

\begin{equation}\label{syst}
\begin{cases}
&x={\rm const}\,,\\
& \dot{z}=-\frac{1}{4}(3a^2+b^2-2iab)\,\frac{xz}{xy-|z|^2}\,,\\
& \dot{y}=\frac{3a^2xy}{xy-|z|^2}\,.
\end{cases}
\end{equation}
In particular if we consider as initial SKT structure
$$
\omega_0=-i\zeta^{1\bar 1}-i\zeta^{2\bar 2}\,,
$$
then the system \eqref{syst} has solutions
$$
\begin{cases}
&x\equiv 1\\
&z\equiv 0\\
& y(t)=3a^2 t
\end{cases}
$$
defined for every $t$ and
$$
\omega(t)=-i\zeta^{1\bar 1}-i3a^2 t\zeta^{2\bar 2}\,.
$$
\end{example}

\section{The pluriclosed flow  as bracket flow}
We regard  the pluriclosed  flow \eqref{flowST} on Lie algebras as a bracket flow on $\R^{2n}$ working as in \cite{lauret2}. The idea consists on studying evolution of brackets instead of  the Hermitian metrics.
We briefly describe the clever trick of \cite{lauret2} adapted to our setting:

\medskip
Let  $(\g,\mu_0,J, g_0, \omega_0)$ be a almost Hermitian Lie algebra.  Then  $(\g,\mu_0,J,\omega_0)$ can be thought as $\R^{2n}$ equipped with the standard Hermitian structure $(J_0,\omega_0, \langle \cdot, \cdot\rangle)$ and  a bracket  $\mu_0$. Consider in this setting the system
\begin{equation}\label{flowRn}
\begin{cases}
\frac{d}{dt}\omega=-(\rho^{B})^{1,1}(\omega)\\
\omega(0)=\omega_0
\end{cases}
\end{equation}
where $\rho^B(\omega)$ is computed with respect to $\omega$ and $\mu_0$ using formulae \eqref{unitary}, i.e.,
\begin{equation}\label{rhoBomega}
\rho^B(\omega)(X,Y)=i \sum_{r=1}^{n}\Big(g(\mu_0(\mu_0(X,Y),Z_r),Z_{\br})-g(\mu_0(Z_r,Z_{\br}),\mu_0(X,Y))\Big)\,,
\end{equation}
$g$ is the inner product induced by $(\omega,J_0)$ and $\{Z_r\}$ is a unitary frame.

\medskip
Let
$$
V:=\Lambda^2(\R^{2n})^*\otimes \R^{2n}
$$
be the space of skew-symmetric $2$-forms on $\R^{2n}$ taking values in $\R^{2n}$ and let
\begin{equation}\label{A}
\mathcal{A}:=\{\mu\in  V\,\,:\,\, \mu \mbox{ satisfies the Jacoby identity and }N_{\mu}=0\}
\end{equation}
where $N_{\mu}$ is the Nijenhuis tensor
$$
N_{\mu}(X,Y):=\mu(J_0X,J_0Y)-J_0\mu(J_0X,Y)-J_0\mu(X,J_0Y)-\mu(X,Y)\,.
$$
$\mathcal{A}$ can be regarded as the space of all $2n$-dimensional Lie algebras equipped with a complex structure.
Given a form $\omega\in \Lambda^{2}(\R^{2n})^*$ compatible with $J_0$, there exists a (non-unique) map $h\in {\rm GL}(n,\C)$ such that
\begin{equation}\label{h}
\omega(\cdot,\cdot)=\omega_0(h\cdot,h\cdot)\,.
\end{equation}
For every $h\in {\rm GL}(n,\C)$ satisfying \eqref{h}, we can define
a new bracket $\mu\in \mathcal{A}$ using the natural relation
$$
\mu:=h\cdot \mu_0\,,
$$
where
$$
h\cdot \mu_0(X,Y)=h\,\mu_0\left(h^{-1}X,h^{-1}Y\right)\,.
$$
Since $h$ belongs to ${\rm GL}(n,\C)$, $h\cdot \mu\in \mathcal{A}$.
Moreover, every $\mu\in \mathcal{A}$ induces a $2$-form $\rho^B_{\mu}\in \Lambda^2(\R^{2n})^*$ defined according to \eqref{theta2} as
$$
\rho^B_{\mu}(X,Y)=i \sum_{r=1}^{n}\Big(\langle\mu(\mu(X,Y),Z_r),Z_{\br}\rangle-\langle\mu(Z_r,Z_{\br}),\mu(X,Y)\rangle\Big)\,,
$$
where $\{Z_r\}$ be the standard unitary basis on $(\R^{2n},J_0)$ .
We denote by $P_{\mu}$ the endomorphism corresponding to $(\rho_{\mu}^B)^{1,1}$ and $\mu$ via $\omega_0$, i.e.
\begin{equation}\label{P}
\omega_0(P_{\mu}(X),Y)=(\rho_{\mu}^{B})^{1,1}(X,Y)\,.
\end{equation}
By definition $P_{\mu}$ is $\omega_0$-symmetric and it commutes with $J_{0}$.

In the same way we denote by $P(\omega)$ the endomorphism corresponding $(\rho^{B})^{1,1}(\omega)$ via $\mu_0$, i.e.
$$
\omega(P(\omega)X,Y)=(\rho^{B})^{1,1}(\omega)\,.
$$
 Note that
$$
P_{\mu_0}=P(\omega_0)\,,
$$
since
$$
\rho^B(\omega_0)=\rho^{B}_{\mu_0}\,.
$$
The following lemma describes as the two endomorphisms $P_{\mu}$ and $P(\omega)$ are related and it can be deduced from the equivalence between the Hermitian Lie algebras $(\mu_0, \omega, J_0)$ and 
$(\mu, \omega_0, J_0)$.
\begin{lemma}\label{lemmalauret}
The following formula holds
\begin{equation}
P_{\mu}=h\, P(\omega)h^{-1}\,.
\end{equation}
\end{lemma}
%\begin{proof} The proof consists of  a straightforward computation.  We have
%$$
%\omega_0(hP(\omega)h^{-1}X,Y)= \omega_0(hP(\omega)h^{-1}X,hh^{-1}Y)=\omega(P(\omega)h^{-1}X,h^{-1}Y)
%$$
%and formula \eqref{rhoBomega}  implies
%$$
%\begin{array}{lcl}
%\omega(P(\omega)h^{-1}X,h^{-1}Y)&=&-i \sum_{r=1}^{n}\Big(g(\mu_0(\mu_0(h^{-1}X,h^{-1}Y),Z_r),Z_{\br})\\[3 pt]
%&&-g(\mu_0(Z_r,Z_{\br}),\mu_0(h^{-1}X,h^{-1}Y))\Big)\,.
%\end{array}
%$$
%$\{Z_r\}$ being a unitary frame with respect to $g$.
%Therefore
%
%\begin{multline*}
%\omega_0(h^{-1}P(\omega)h^{-1}XY)=\\
%i\sum_{r=1}^{n}\Big(g(\mu_0(h^{-1}\mu(X,Y),Z_r),Z_{\br})-g(\mu_0(Z_r,Z_{\br}),h^{-1}\mu(X,Y))\Big)\\
%\,i \sum_{r=1}^{n}\Big(g(h^{-1}\mu(\mu(X,Y),hZ_r),h^{-1}hZ_{\br})-g(h^{-1}\mu(hZ_r,hZ_{\br}),h^{-1}\mu(X,Y))\Big)\\
%=\,i \sum_{r=1}^{n}\Big(\langle \mu(\mu(X,Y),hZ_r),hZ_{\br}\rangle -\langle(h^{-1}\mu(hZ_r,hZ_{\br}),\mu(X,Y)\rangle\Big)\\
%=\,\rho^B_{\mu}(X,Y)\,,
%\end{multline*}
%
%as required.
%\end{proof}

We further consider the bracket flow
\begin{equation}\label{bracketflow}
\begin{cases}
\frac{d}{dt}\mu=\frac12\, \delta_{\mu}(P_{\mu})\\
\mu(0)=\mu_0
\end{cases}
\end{equation}
where $\delta_{\mu}\colon\mathfrak{gl}_n(\C)\to V_{2n}$ is defined by
$$
\delta_{\mu}(\alpha)=\mu(\alpha\cdot,\cdot)+\mu(\cdot,\alpha\cdot)-\alpha\mu(\cdot,\cdot)\,.
$$

\begin{theorem}\label{lauret}
Let $\omega(t)$ be a solution to \eqref{flowRn} and let $\mu(t)$ be a solution to \eqref{bracketflow}. Then there exists a curve $h=h(t)\in {\rm GL}(n,\C)$ such that:
\begin{enumerate}
\item[1.] $\omega=h^*\omega_0$;

\vspace{0.1cm}
\item[2.] $\mu=h\cdot \mu_0$;

\vspace{0.1cm}
\item[3.] $\frac{d}{dt}h=-\frac12hP(\omega)=-\frac12 P_{\mu}h$.
\end{enumerate}
\end{theorem}
\begin{proof}
Let $\omega=\omega(t)$ be a solution to \eqref{flowRn} and let
$h(t)$ be the solution to the linear ODE
$$
\begin{cases}
\frac{d}{dt}h=-\frac12h P(\omega)\\
h_0={\rm I}\,.
\end{cases}
$$
If $\tilde\omega=h^*\omega_0$, then
$$
\begin{aligned}
\frac{d}{dt}\tilde \omega(\cdot,\cdot)&=\omega_0(h'\cdot,h\cdot)+\omega_0(h\cdot,h'\cdot)=-\frac12 \omega_0(hP(\omega)\cdot,h\cdot)-\frac12 \omega_0(h\cdot,hP(\omega)
\cdot)\\
&=-\frac12\tilde\omega(P(\omega)\cdot,\cdot)-\frac12\tilde\omega(\cdot,P(\omega)\cdot)
\end{aligned}
$$
Since
$$
\frac{d}{dt}\omega(\cdot,\cdot)=-\frac12\omega(P(\omega)\cdot,\cdot)-\frac12\omega(\cdot,P(\omega)\cdot)
$$
$\omega$ and $\tilde \omega$ solve the same ODE
$$
\begin{cases}
\frac{d}{dt}\beta(\cdot,\cdot)=-\frac12\beta(P(\omega)\cdot,\cdot)-\frac12\beta(\cdot,P(\omega)\cdot)\\
\beta(0)=\omega_0
\end{cases}
$$
and therefore $\omega=h^*\omega_0$. Let $\lambda=h\cdot\mu_0$.  Using Lemma \ref{lemmalauret} and
$h'=-\frac12P_\lambda h$ we obtain
$$
\begin{aligned}
\lambda'(\cdot,\cdot)&=h'\mu_0(h^{-1}\cdot,h^{-1}\cdot)-h\mu_0(h^{-1}h'h^{-1}\cdot,h^{-1}\cdot)-h\mu_0(h^{-1}\cdot,h^{-1}h'h^{-1}\cdot)\\
&= \left(h'h^{-1}\right)h\mu_0(h^{-1}\cdot,h^{-1}\cdot)-h\mu_0(h^{-1}(h'h^{-1})\cdot,h^{-1}\cdot)-h\mu_0(h^{-1}\cdot,h^{-1}(h'h^{-1})\cdot)\\
&=-\delta_{\lambda}(h'h^{-1})=\frac12 \delta_{\lambda}(P_\lambda)\,.
\end{aligned}
$$
Therefore $\lambda$ and $\mu$ solve the same ODE and consequently they coincide, as required.
\end{proof}

\begin{oss} Note that the Bismut scalar form $b(\omega)=g(\rho^B(\omega),\omega)$ reads in terms of bracket as
$$
b_{\mu}:=\langle \rho^B_\mu,\omega\rangle=\sum_{r,k}\langle\mu(Z_r,Z_{\br}),\mu(Z_k,Z_{\bar k})\rangle\,,
$$
i.e.,
$$
b_{\mu}=-\left\|\sum_r \mu(Z_r,Z_{\br}) \right\|^2\,.
$$
$\{Z_r\}$ being an arbitrary unitary frame.
\end{oss}

\begin{lemma}
The bracket flow \eqref{bracketflow} preserves the center of $\mu_0$.
\end{lemma}
\begin{proof}
Consider on $(\R^{2n},J_0, \omega_0)$ an arbitrary $\mu_0\in \mathcal{A}$. Let $\xi_{0}$ be the center of $\mu_0$ and $\xi_0^{\perp}$ its orthogonal complement with respect to $\langle\,,\,\rangle$.
Every $J_0$-compatible non-degenerate form $\omega$ can be decomposed as $\omega=\alpha'+\alpha$  with respect the splitting $\g=\xi\oplus \xi^{\perp}$, where $\alpha$ is the restriction of $\omega$ to $\xi^{\perp}\times\xi^{\perp}$. We can write in particular $\omega_0=\alpha_0+\alpha_0'$.
Formula \eqref{rhoBomega} implies that  $\rho^B(X^{\xi},\cdot)$ vanishes for every $X^{\xi}\in \xi$. Therefore the solution $\omega(t)$ to \eqref{flowRn} can be written as $\omega(t)=\alpha'_0+\alpha(t)$ and every $h=h(t)$ satisfying conditions 1,2,3 of Theorem \ref{lauret} preserves $\alpha_0'$, i.e. $h(t)^*(\alpha'_0)=\alpha'_0$ for every $t$. There follows $h(t)(\xi_0)=\xi_0$ for any $t$. The solution $\mu(t)$ to the bracket flow \eqref{bracketflow} is defined in terms of $h$ and $\mu_0$ as $\mu(t)=h(t)\cdot \mu_0$ and for every $t$ the kernel of $\mu(t)$ is $\xi_{t}=h(t)\xi_0$. Hence $\xi_t\equiv \xi_0$, as required.
\end{proof}

We describe  now how the SKT condition reads in terms of brackets:\\
Let $\mu$ be a bracket in $\mathcal A$.  Then $\mu$ induces the differential operator
$$
d_{\mu}\colon \Lambda^{r} (\R^{2n})^*\to \Lambda^{r+1}(\R^{2n})^*
$$
defined in terms of $\mu$ as
$$
d_{\mu}\gamma(X_0,X_1,\dots, X_r)=\sum_{0\leq i\leq j\leq r} (-1)^{i+j}\gamma(\mu(X_i,X_j),X_0,\dots, \hat{X_i},\dots, \hat{X}_{j},\dots, X_r)\,.
$$
Furthermore, the complex extension of $d_{\mu}$ splits with respect to $J_0$ as $d_{\mu}=\partial_{\mu}+\bar{\partial}_{\mu}$. We denote by $\de$, $\bar \de$ the usual differential complex operator on $(\R^{2n},J_0)$.
Every SKT Lie algebra can be seen as  a Hermitian Lie algebra $(\R^{2n},J_0,\langle\,,\rangle,\mu)$ whose fundamental form $\omega_0$ satisfies
$$
\bar \de_{\mu}\de_{\mu}\omega_0=0\,.
$$
This motivates the following
\begin{definition}
A bracket  $\mu\in\mathcal{A}$ is {\em SKT} if
\begin{equation}\label{SKT}
\bar{\partial}_{\mu}\partial_{\mu}\omega_0=0\,.
\end{equation}
\end{definition}
The identity \eqref{SKT} is an algebraic equation in $\mu$ and, therefore, the set of SKT brackets gives an algebraic subset of $\mathcal{A}$.

\section{Long time existence for nilmanifolds }
In this subsection we focus on SKT structures on  nilpotent Lie algebras proving Theorem \ref{main1}.

  The following two results will be important in the sequel

\begin{theorem} $($\cite{EFV}$)$
Let $(\g,\mu,J,\omega)$ be an SKT nilpotent Lie algebra.  Then $\g$ is $2$-step and $J$ preserves the center $\xi$  of $\g$.
\end{theorem}

\begin{theorem} $($\cite{Vezzoni}$)$
For a $2$-step nilpotent almost Hermitian Lie algebra, the Chern form $\rho^C$ is always vanishing.
\end{theorem}
Therefore in the SKT nilpotent case we have to handle  $2$-step nilpotent Lie algebras.  Using the general formula \eqref{fundamental2}
$$
\rho^B=\rho^C-dd^*\omega
$$
we have that the pluriclosed flow  reduces in this case  to
\begin{equation}\label{2step}
\begin{cases}
\frac{d}{dt}\omega=(dd^*\omega)^{1,1}\,,\\
\omega(0)=\omega_0\,.
\end{cases}
\end{equation}

Let us consider then  an SKT ($2$-step) nilpotent Lie algebra $(\g,\mu,J,\omega)$ with induced metric $g$.
We denote by $\sharp\colon \g\to \g^*$  the duality induced by the inner product  $g$. Given a vector subspace $W$ of $\g$ we set $W^{\sharp}:=\sharp(W)$ and we denote by $W^{\perp}$ the orthogonal complement of $W$ with respect to $g$. Finally we denote by $\theta = - J d^* \omega$ the Lee form of $(J,\omega)$.
We have the following

\begin{proposition} The Lee form $\theta$  of a  nilpotent  SKT Lie algebra $({\mathfrak g}, \mu, J, g)$ belongs to  $(J \frak g^1)^{\sharp} \subseteq \xi^{\sharp}$.
\end{proposition}
\begin{proof}
By definition of  $\theta$ we have
$$
g (J \theta, \alpha) = g (\omega, d \alpha),
$$
for every $\alpha \in \frak g^*$. Note that if $\alpha \in ((\frak g^1)^{\perp})^\sharp$, then $d \alpha =0$. Therefore, $J \theta \in (\frak g^1)^\sharp$, i.e. $\theta \in (J \frak g^1)^\sharp$.
On the other hand,  since $\frak g$ is $2$-step nilpotent and  the center $\xi$ is $J$-invariant,  it follows  that $J \frak g^1 \subseteq \xi$, i.e. that  $\theta \in \xi^\sharp$.
\end{proof}

\begin{lemma}\label{anna1}
For a nilpotent SKT Lie algebra $(\mathfrak g, \mu, J, g)$ the $(1,1)$-part   $(ric^B)^{1,1}$ of the Ricci tensor of the Bismut connection  is symmetric and it is related to $\rho^B$ by
$$
(\rho^B)^{1,1}(X,Y)=(ric^B)^{1,1}(X,JY),
$$
for every $X,Y\in \g$.
\end{lemma}
\begin{proof}
We can write
$$
\g=\frak \xi \oplus \xi^{\perp},
$$
where $\xi$ is the center of $\g$. The $2$-step condition implies
$$
[\xi^{\perp} , \xi^{\perp}] \subseteq \xi\,.
$$
Every $X \in \frak g$ can be written accordingly as
$$
X = X^{\xi} + X^{\perp},
$$
where $X^{\xi} \in \xi$ and $X^{\perp}  \in \xi^{\perp}$. By \cite[Lemma 2.1]{Enrietti}  we have that
$$
\nabla^B_{X^{\xi}} Y^{\xi}  =0, \quad \nabla^B_{X^{\xi}} Y^{\perp}  \in \xi^{\perp}, \quad \nabla^B_{X^{\perp}} Y^{\xi}  \in \xi^{\perp}$$
and
$$
\nabla^B_{X^{\perp}} Y^{\perp}  = \frac{1}{2} ([X^{\perp}, Y^{\perp}] - [JX^{\perp}, J Y^{\perp}] ) \in \xi.
$$
Therefore
$$
\theta (\nabla^B_X Y) = \frac 12 \theta  ([X^{\perp}, Y^{\perp}] - [JX^{\perp}, J Y^{\perp}] ).
$$
Using the integrability of $J$ we get
$$
\begin{aligned}
\theta (\nabla^B_X Y)=& \frac 12 \theta  (-J[J X^{\perp}, Y^{\perp}] - J [X^{\perp}, JY^{\perp}] )\\
=& \frac 12 (J \theta) ([J X^{\perp}, Y^{\perp}] +  [X^{\perp}, J Y^{\perp}]) \\
=& - \frac 12 d (J \theta) (J X^{\perp}, Y^{\perp}) - \frac 12 d (J \theta) (
X^{\perp}, J Y^{\perp}).
\end{aligned}
$$
Taking into account that  $\rho^B = \rho^C + d (J \theta)$ and that $\rho^C$ in our cases vanishes, we have
$$
\theta (\nabla^B_X Y) = - \frac 12  \rho^B  (J X^{\perp}, Y^{\perp}) - \frac 12 \rho^B  (X^{\perp}, JY^{\perp})\,.
$$
Therefore
$$
(\nabla^B_X \theta) (J Y)=-\theta (\nabla^B_X J Y)= \frac 12  \rho^B  (J X^{\perp}, JY^{\perp}) - \frac 12 \rho^B  (X^{\perp}, Y^{\perp})=-(\rho^B)^{2,0+0,2}(X^{\perp},Y^{\perp}).
$$
Using
$$
\rho^B (X, Y) = ric^B (X, JY) + (\nabla^B_X \theta) (J Y) = ric^B (X, JY) -(\rho^B)^{2,0+0,2}(X^{\perp},Y^{\perp}),
$$
we get
$$
(\rho^B)^{1,1}(X,Y)=\frac{1}{2} [ric^B (X, JY) - ric^B (JX, Y)] = (ric^B)^{1,1} (X, JY)\,,
$$
as required.
\end{proof}
\begin{theorem}\label{anna2}
For a  nilpotent  SKT Lie algebra $({\mathfrak g},  \mu, J, g)$ we have
$$
(ric^B)^{1,1}  (X, Y) = 2 (ric^g)^{1,1}  (X^{\perp}, Y^{\perp})\,,
$$
for every $X,Y\in \g$.
\end{theorem}
\begin{proof}
Using formula \eqref{rhonilpotent} and Lemma \ref{anna1} we have that
$$
(ric^B)^{1,1}  (X, Y)=(ric^B)^{1,1}  (X^{\perp}, Y^{\perp})\,.
$$
In view of formula \eqref{fundamental} 
%the two Ricci tensors $ric^B$ and $ric^g$ are related by
%$$
%ric^B (X, Y) = ric^g (X, Y) - \frac 12   (d^* c) (X,Y) - \frac 14 \sum_{i = 1}^{2n} g (T^B(X, e_i), T^B (Y, e_i)),
%$$
%where $c$  and $T$ torsion $3$-form and the torsion tensor of the Bismut connection.  Since $d^*c$ is a 2-form, the previous 
Lemma \ref{anna1} implies
%$$
%(d^*c)^{1,1}=0
%$$
%and therefore we get
$$
(ric^B)^{1,1}  (X, Y)=(ric^g)^{1,1} (X, Y) -  \frac 18 \sum_{i = 1}^{2n} [g (T^B (X, e_i), T^B(Y, e_i) +  g (T^B(JX, e_i), T^B (JY, e_i)]\,. $$
Hence the claim consists on showing that
$$
\frac18\,\sum_{i = 1}^{2n} [g (T^B (X^{\perp}, e_i), T^B (Y^{\perp}, e_i) +  g (T^B (JX^{\perp}, e_i), T^B (JY^{\perp}, e_i)]=-2(ric^g)^{1,1} (X^{\perp}, Y^{\perp})\,.
$$
Let $S\colon \xi^{\perp}\to \xi^{\perp}$ be the symmetric operator defined by the relation
$$
g(S(X^{\perp}),Y^{\perp})=ric^g(X^{\perp},Y^{\perp})\,.
$$
By \cite{Eberlein}  we have that $S$ can be written as
$S = \frac 12  \sum_{i = 1}^{2p} \iota (z_i)^2,$
where $\{z_1, \ldots, z_{2p}\}$ is an orthonormal basis of $\xi$ and  the skew-symmetric map $\iota(Z):
\xi^{\perp} \to \xi^{\perp}$ is defined by
$$
g(\iota (Z) X, Y) = g( [X, Y], Z), \quad X, Y \in \xi^{\perp},
$$
for $Z \in \xi.$ Equivalently $\iota (Z) (X) = - (ad_X)^* (Z)$, for every $X \in \xi^{\perp},$ where $(ad_X)^*$ is the adjoint of $ad_X$ with respect to the inner product $g$. In particular $S$ is negative definite on $\xi^{\perp}$ and
$$
ric^g (X^{\perp}, Y^{\perp}) = \frac{1}{2} \sum_{i = 1}^{2p}  g( \iota (z_i)^2 (X^{\perp}), Y^{\perp}) = - \frac{1}{2} \sum_{i = 1}^{2p} g (\iota (z_i) X^{\perp}, \iota (z_i) Y^{\perp}),
$$
where $(z_1, \ldots,  z_{2p})$ is an  orthonormal basis of $\xi$.
By a direct computation we have that, for every  $z_i \in \xi$,
%\begin{array}{l}
%g (\nabla^B_{X^{\xi}} Y^{\perp}, Z) = - \frac 12 g( [Y^{\perp}, Z] + [J Y^{\perp}, J Z], X^{\xi}), \\[3 pt]
%g (\nabla^B_{Y^{\perp}} X^{\xi}, Z) = - \frac 12 g( [Y^{\perp}, Z] - [J Y^{\perp}, J Z], X^{\xi}),
%\end{array}
%$$
%we have that, for every  $z_i \in \xi$,
$$
g (T^B(X^{\perp},  z_i), T^B  (Y^{\perp}, z_i)) = g(\iota (z_i) (J Y^{\perp}), \iota(z_i) (J X^{\perp}))=- 2 ric^g (J X^{\perp}, J Y^{\perp}) .
$$
%$$
%\begin{array}{lcl}
%g (T^B(X^{\perp},  z_i), T^B  (Y^{\perp}, z_i)) &= &g([JX^{\perp}, J(T^B (Y^{\perp}, z_i)], z_i) = g ( \iota (z_i)  (JX^{\perp}), JT^B (Y^{\perp}, z_i))\\[3 pt]
%&=& - g (J  \iota (z_i)  (JX^{\perp}), T^B (Y^{\perp}, z_i))=  - g ([J Y^{\perp}, J^2  \iota (z_i) (JX)], z_i)\\[3 pt]
%&=& g(\iota (z_i) (J Y^{\perp}), \iota(z_i) (J X^{\perp})) .
%\end{array}
%$$
%Therefore
%\begin{equation}\label{prima}
%\sum_{i = 1}^{2p} g (T^B(X^{\perp},  z_i), T^B  (Y^{\perp}, z_i))  =  - 2 ric^g (J X^{\perp}, J Y^{\perp}).
%\end{equation}
On the other hand, if $v_i \in \xi^{\perp}$, then
$$
g (T^B(X^{\perp},  v_i), T^B  (Y^{\perp}, v_i)) = g ([J v_i, J X^{\perp}], [J v_i, J Y^{\perp}]).
$$
By  Section 2 in \cite{Eberlein} for a  metric $2$-step nilpotent Lie algebra one has
%$$
%R^g (X^{\perp}, Y^{\perp}) X^{\perp} = \frac{3}{4} \iota ([X^{\perp}, Y^{\perp}] )X^{\perp},
%$$
%for every $X^{\perp}, Y^{\perp},\in \xi^{\perp} $ and consequently
$$
g(R^g (X^{\perp}, Y^{\perp}) X^{\perp}, W) = \frac 34 g (\iota ([X^{\perp}, Y^{\perp}] )X^{\perp}, W) = \frac 34 g([X^{\perp}, Y^{\perp}], [X^{\perp}, W])
$$
for every $W\in \mathfrak{g}$
 and $X^{\perp}, Y^{\perp},\in \xi^{\perp} $.
As a consequence, for every $v_i \in \xi^{\perp}$ we have
$$
g (T^B(X^{\perp},  v_i), T^B  (Y^{\perp}, v_i)) = \frac 43 g (R^g (J v_i, J X^{\perp}) J v_i, J Y^{\perp}).
$$
Moreover, by \cite[pag. 622]{Eberlein}
$$
\sum_{i =1}^{2n - p} g (R^g (v_i, X^{\perp}) Y^{\perp}, v_i) = \frac 34 \sum_{k = 1}^{2p} g (\iota (z_k)^2 X^{\perp}, Y^{\perp}).
$$
Therefore
$$
\sum_{i = 1}^{2n -p} g (T^B(X^{\perp},  v_i), T^B  (Y^{\perp}, v_i)) =  -  \sum_{k = 1}^{2p} g (\iota (Jz_k)^2 JX^{\perp}, JY^{\perp}) =  - 2 ric^g (J X^{\perp}, J Y^{\perp}).
$$
In this way
$$
(ric^B)^{1,1}  (X^{\perp}, Y^{\perp}) =   2 (ric^g)^{1,1}  (X^{\perp}, Y^{\perp}).
$$
\end{proof}

\begin{oss}
By \cite{Eberlein}  for  a  metric  $2$-step Lie algebra   $(\g, \mu, g)$  one has
 \begin{enumerate}
\item[1.] $ric^g (X, Z) = 0$ for all $X \in \xi$ and  $Z \in \xi^{\perp}$.
\item[2.] $ric^g (Z, Z^*) = -  \frac 14  {\rm tr} \,  \iota (Z) \iota ( Z^*) $ for $Z,Z^* \in \xi$.   In particular
$ric^g (Z,  Z) \geq 0$ for all $Z  \in \xi$  with equality if and only if  $ \iota(Z) = 0$.
\end{enumerate}
Moreover, giving a Hermitian structure $(g,J)$ on $\g$,
for $X\in \xi$ and $Y \in \xi^{\perp}$ we have
$$
g (T^B (X, e_i), T^B (Y, e_i)) = g (\nabla^B_X e_i - \nabla^B_{e_i} X, \nabla^B_Y e_i - \nabla^B_{e_i} Y - [Y, e_i])\,.
$$
If $e_i \in \xi$ then $$g (T^B (X, e_i), T^B (Y, e_i)) =0.$$ If $e_i \in \xi^{\perp}$ we have that $\nabla^B_X e_i - \nabla^B_{e_i} X \in \xi^{\perp}$ and $\nabla^B_Y e_i - \nabla^B_{e_i} Y - [Y, e_i] \in \xi$. So again  $g (T^B (X, e_i), T^B (Y, e_i))=0$ and summing up
$$
 \frac 14 \sum_{i = 1}^{2n} g (T^B (X, e_i), T^B (Y, e_i)) =0
$$
for every $X\in \xi$ and $Y^{\perp}\in \xi^{\perp}$.
There follows that
$$
(ric^B)^{1,1}  (X, Y^{\perp}) = 2 (ric^g)^{1,1}  (X^{\perp}, Y)
$$
for all $X\in \g$ and $Y\in \xi^{\perp}$, while
$$
(ric^B)^{1,1} _{\xi\times \xi}\neq 2 (ric^g)^{1,1}_{\xi\times\xi}\,.
$$

\end{oss}

\medskip
Let us consider now the space  $\mathcal N$ of all $2n$-dimensional nilpotent  Lie algebras equipped with a complex structure. Such a space can be seen as a subspace of the space $\mathcal A$ defined in \eqref{A}.   Let $\mu_0$ be an SKT bracket in $\mathcal{N}$ and let $\mu(t)$ be the solution to \eqref{bracketflow} satisfying $\mu(0)=\mu_0$. Then $\mu(t)$ is SKT for every $t$
and we have
$$
\frac{d}{dt}\langle \mu,\mu\rangle=2 \left\langle \frac{d}{dt}\mu,\mu\right\rangle=\left\langle \delta_{\mu}(P_\mu),\mu\right\rangle=-4\left\langle P_\mu,Ric_{\mu}\right\rangle
$$
where
$$
Ric_{\mu}=-\frac12 \sum_{i=1}^{2n}({\rm ad}_{\mu}E_i)^t{\rm ad}_{\mu}E_i+\frac14 \sum_{i=1}^{2n}{\rm ad}_{\mu}E_i({\rm ad}_{\mu}E_i)^{t} 
$$
is the usual Ricci operator induced by $\mu$ (see \cite[Lemma 4.2]{lauret2}) (here $\{E_k\}$ is the canonical basis of $\R^{2n}$). Using that $P_{\mu}$ is of type $(1,1)$ (i.e. that it commutes with $J_0$) and that $P_{\mu}$ vanishes along the center of $\mu$
$$
\xi_{\mu}=\left\{X\in \R^{2n}\,\,\mbox{ s.t. } \mu(X,Y)=0\quad \mbox{ for all }Y\in \R^{2n}  \right\}
$$
we have
$$
\frac{d}{dt}\langle \mu,\mu\rangle=-4\sum_k\left\langle P_{\mu}(e_k),(Ric_{\mu})^{1,1}(e_k)\right\rangle
$$
where $\{e_k\}$ is an arbitrary orthonormal basis to $\xi_\mu^{\perp}$. 
On the other hand Lemma \ref{anna1} and Theorem \ref{anna2}  imply
$$
\sum_k\left\langle P_{\mu}(e_k),(Ric_{\mu})^{1,1}(e_k)\right\rangle=2\left\langle P_\mu,P_{\mu}\right\rangle
$$
and so 
\begin{equation}
\label{enstimate}
\frac{d}{dt}\langle \mu,\mu\rangle=-8\langle P_{\mu},P_{\mu}\rangle\leq 0\,,
\end{equation}
which readily implies that in the nilpotent case the unique solution to the system \eqref{bracketflow} is defined for every positive $t$. This fact, together with Theorem \ref{lauret}, implies the statement of Theorem \ref{main1}.

\medskip

Moreover, we have the following
\begin{proposition}
In the nilpotent SKT case the maximal solution to \eqref{bracketflow} converges to the abelian bracket.
\end{proposition}
\begin{proof}
Let $\mu(t)$ be the maximal solution to \eqref{bracketflow}. We prove that $\|\mu(t)\|^2$ tends to zero when $t$ tends to infinity.  In view of \eqref{enstimate} we have 
$$
\frac{d}{dt}\langle \mu,\mu\rangle=-8\langle P_{\mu},P_{\mu}\rangle=-2 \sum_{k}\langle Ric_{\mu}^{1,1}(e_k),Ric_{\mu}^{1,1}(e_k)\rangle\leq -2 \left(\sum_k \langle Ric^{1,1}_\mu(e_k),e_k \rangle \right)^2 
$$
where $\{e_k\}$ is an arbitrary orthonormal basis of $\xi_{\mu}^{\perp}$. Since $\xi_\mu^{\perp}$ is $J_0$-invariant, then 
$$
\sum_k \langle Ric^{1,1}_\mu(e_k),e_k \rangle =\sum_k \langle Ric_\mu(e_k),e_k \rangle
$$
i.e., 
$$
\frac{d}{dt}\langle \mu,\mu\rangle\leq -2\left(\sum_k \langle Ric_\mu(e_k),e_k \rangle\right)^2
$$
From the definition of $Ric_\mu$ and taking into account that the $e_k$'s belong to the orthogonal complement $\xi^{\perp}_{\mu}$ to the center of $\mu$, we have 
$$
\frac{d}{dt}\langle \mu,\mu\rangle\leq -2\left(\sum_k \langle Ric_\mu(e_k),e_k \rangle\right)^2=-\frac{1}{2}\left(\sum_{i,k}\langle \mu(e_i,e_k),\mu(e_i,e_k)\rangle\right)^2=-\frac{1}{2}\|\mu\|^4
$$
and the claim follows. 
\end{proof}

\begin{example}
Here we study the basic Example \ref{KT} in terms of bracket flow.
The starting bracket is
$$
\mu_0=-\frac12\, \zeta^1\wedge\zeta^{\bar 1}\otimes Z_2+\frac12\, \zeta^1\wedge\zeta^{\bar 1}\otimes Z_{\bar 2}
$$
which corresponds to the bracket of the Lie algebra $\mathfrak{h}_3\oplus \R$.  Since the bracket flow  preserves the center, we look for a solution $\mu$ to \eqref{bracketflow}  taking value only at $(Z_1,Z_{\bar 1})$, i.e.
$$
\mu=\mu_{1\bar 1}^2\, \zeta^1\wedge\zeta^{\bar 1}\otimes Z_2+\mu_{1\bar 1}^{\bar 2}\, \zeta^1\wedge\zeta^{\bar 1}\otimes Z_{\bar 2}\,.
$$
For such a bracket we have
$$
\rho^B_{\mu}=-2i\,|\mu_{1\bar 1}^2|^2\, \zeta^1\wedge\zeta^{\bar 1}
$$
and
$$
P_{\mu}=-2\,|\mu_{1\bar 1}^2|^2\, \zeta^1\otimes Z_{ 1}+2\,|\mu_{1\bar 1}^2|^2\, \zeta^{\bar 1}\otimes Z_{\bar 1}\,.
$$
Therefore
$$
\delta_{\mu}(P_{\mu})(Z_1,Z_{\bar 1})=2\mu(P_{\mu}(Z_1),Z_{\bar 1})=-4 |\mu_{1\bar 1}^2|^2 \mu(Z_1,Z_{\bar 1})
$$
and the corresponding bracket flow equation is
\begin{equation}\label{z}
\begin{cases}
\dot{z}=-2 |z|^2\,z,\\
z(0)=-\frac12
\end{cases}
\end{equation}
where $z=\mu_{1\bar 1}^2$.  Since \eqref{z} has as solution the real function
$$
z(t)=-\frac{1}{2(t+1)^{\frac12}}
$$
the solution $\mu(t)$ of the bracket flow is defined for every positive $t$ and converges in $\mathcal{A}$ to the null bracket corresponding to the abelian Lie algebra.
\end{example}

\section{Evolution of  Tamed Symplectic forms on a complex manifold}

Let $(M,J)$ be a complex manifold. We recall that a symplectic form $\Omega$ on $M$ {\em tames} $J$ if
\begin{equation}\label{hs}
\Omega(JX,X)>0
\end{equation}
for every non-zero tangent vector field $X$ on $M$. Such a condition is weaker than the compatibility of $\Omega$ with $J$ since in this case the positive tensor induced by \eqref{hs} is not symmetric. A structure $(J,\Omega)$
composed by a complex structure and a taming symplectic form was  called in \cite{ST} a {\em Hermitian-symplectic} structure. Such a structure arises  considering static solutions of  the pluriclosed flow \eqref{flowST}. Indeed if an SKT form $\omega$ satisfies the Hermitian-Einstein equation $r\,\omega=(\rho^B)^{1,1}(\omega)$ with $r\in \R$ and  $r\neq 0$, then $\Omega=\frac{1}{r}\rho^B$ is a symplectic form taming $J$.

 In \cite{EFV} it was observed that Hermitian-symplectic structures are actually special SKT structures. This is because  given a symplectic form $\Omega$
taming
$J$ and considering  the decomposition of $\Omega$ in complex-type forms
$$
\Omega=\omega +\beta+\bar{\beta}\in \Lambda^{1,1}\oplus \Lambda^{2,0}\oplus \Lambda^{0,2}
$$
one has that $d\Omega$ vanishes if and only if $\beta$ solves
\begin{equation}\label{beta}
\begin{cases}
\bar{\partial}\Omega^{11}=-\partial \beta\\
\bar{\partial}\beta=0\,.
\end{cases}
\end{equation}
In the sequel  of the paper  we are going to take into account the  following evolution equation
\begin{equation}\label{eq:HSflow}
 \begin{cases}
  \frac{d}{d t}\Omega=-\rho^B (\omega)\\
  \Omega(0)=\Omega_0,
 \end{cases}
\end{equation}
which we will call the \em{Hermitian-symplectic} (or simply \em{HS}) flow.

\begin{proposition}\label{HS}
 Let $\Omega_0$  be a   tamed symplectic form on a compact complex manifold $(M,J)$. Then short-time existence of a solution $\Omega(t)$ of \eqref{eq:HSflow} is guaranteed. Moreover, $\Omega(t)$ is a symplectic form taming $J$ for every $t$.
\end{proposition}

\begin{proof}
We can write $\Omega_0=\omega_0+\beta_0+\bar\beta_0$ and the Hermitian-symplectic flow decomposes in its $(1,1)$-part
\begin{equation}\label{eq:SKTflow}
  \begin{cases}
  \frac{d}{d t}\omega=-(\rho^B)^{1,1}(\Omega^{1,1}) \\
  \omega(0)=\omega_0
  \end{cases}
\end{equation}
and the $(2,0)$-part
\begin{equation}\label{eq:betaFlow}
 \begin{cases}
\frac{d}{d t}\beta=-(\rho^B)^{2,0}(\omega) \\
\beta(0)=\beta_0. \\
\end{cases}
\end{equation}
Since \eqref{eq:SKTflow} is the \lq\lq usual'' pluriclosed flow, it admits  a solution $\omega(t)$ defined in an interval $[0,\varepsilon)$, for $\varepsilon$ small enough. On the other hand, since $(\rho^B)^{2,0}(\omega)$ does not depend on $\beta$,
$$
\beta(t)=\beta_0+\int_{0}^t (\rho^B)^{2,0}(\omega)(s)\,ds
$$
is a solution to \eqref{eq:betaFlow} and $\Omega(t):=\omega(t)+\beta(t)$ provides the unique solution to \eqref{eq:HSflow}.

We finally observe that the taming  condition is preserved by the flow. Indeed, $\omega(t)$ is positive since it is a solution to the pluriclosed flow  and $\Omega(t)$ is closed
since
\[
\frac{d}{dt}\left(d\Omega(t)\right) = d\, \left ( \frac{d}{dt}\Omega(t) \right )= -d\rho^B = 0,
\]
and then $d\Omega(t)$ is constant.
\end{proof}

The previous result says that the pluriclosed flow preserves the Hermitian-symplectic condition.
Indeed, a Hermitian-symplectic structure can be defined as an SKT structure $(\omega_0,J)$ together a solution $\beta$ to \eqref{beta}.  As a consequence of Proposition \ref{HS}  we have that if  an SKT form $\omega_0$ admits a solution $\beta_0$ to \eqref{beta}, then the solution  $\omega(t)$ to the pluriclosed flow with initial condition $\omega_0$ has  a solution $\beta(t)$ for every $t$.

We recall the following stability theorem for the Hermitian curvature flow \eqref{HCF} obtained by Streets and Tian

\begin{theorem} $($\cite{ST2}$)$ \label{st} Let $(M, \tilde g, J)$ be a complex manifold with a K\"ahler-Einstein metric
$\tilde g$ and $c_1(M) < 0$ or $c_1(M) = 0$. Then there exists $\epsilon =\epsilon(\tilde g)$ so that if $g_o$ is a $J$-Hermitian metric on $M$  and $\|\tilde{g} - g_o\|_{C^{\infty}} < \epsilon$, then the solution to
\eqref{HCF} with initial condition $g_o$ exists for all time and converges {\em exponentially } to a K\"ahler-Einstein metric.
\end{theorem}
\begin{corollary}
In the hypothesis of Theorem $\ref{st},$ let $\Omega_o$ be a symplectic form on $M$ taming $J$ and such that $\|\tilde{g} - g_o\|_{C^{\infty}} < \epsilon$, where $g_o$ is the Hermitian metric of $\Omega^{1,1}_o$. Then the solution $\Omega(t)$ of flow \eqref{eq:HSflow} with initial condition $\Omega(0)=\Omega_o$ is defined for every $t\in [0,\infty)$ and  converges to a symplectic form whose $(1,1)$-component induces a K\"ahler-Einstein metric.
\end{corollary}
\begin{proof}
Let $\omega_o=\Omega_o^{1,1}$. Then using Theorem \ref{st} we have that the equation
\begin{equation*}
  \begin{cases}
  \frac{d\omega}{d t}=-(\rho^B)^{1,1}(\omega) \\
  \omega(0)=\Omega_o^{1,1}
  \end{cases}
\end{equation*}
has a unique solution $\omega(t)$ defined in $[0,\infty)$ and converging exponentially to a K\"ahler-Einstein structure $\omega_{\infty}$. Since $\omega(t)$ is defined in $[0,\infty)$, the system
\begin{equation*}
 \begin{cases}
  \frac{d\beta}{d t}=-(\rho^B)^{2,0}(\omega) \\
  \beta(0)=\Omega_o^{2,0}
 \end{cases}
\end{equation*}
has a solution $\beta(t)$ in $[0,\infty)$ which can be written as
$$
\beta(t)=\int_0^{t}f(s)\,ds+\Omega_o^{2,0}
$$
$f(s)$ being
$$
f(s)=-(\rho^B)^{2,0}(\omega(s))\,.
$$
We claim that  $f(s)$ converges exponentially to $0$. This last assertion can be proved as follows: Let $g$ be an arbitrary $J$-Hermitian metric on $(M,J)$ with fundamental form $\omega$. 
Then a standard computation yields that in local complex coordinates we have
$$
(\rho^B)^{2,0}(\omega)=-\frac{i}{2}\,\partial_{z^a}(g^{k\bar l}(g_{b\bar l,k}-g_{k\bar l,b}))\, dz^a\wedge dz^b \,,
$$  
Therefore we have the estimates  
$$
\|(\rho^B)^{2,0}(\omega)\|_{C^k}\leq C_k\,\sum_{i+j=k+1}\|\omega\|_{C^{i}}\|\partial\omega\|_{C^{j}}
$$
where all the $C^k$-norms are computed with respect to $\tilde g$. 
Now, since  $\omega(t)$ converges exponentially to $\omega_{\infty}$ and $\omega_{\infty}$ is closed, we have that 
$\partial \omega(t)$ converges exponentially to $0$ in the $C^{\infty}$-norm.  On the other hand 
$$
\|\omega(t)\|_{C^k}\leq \tilde C_k e^{-\lambda_kt}+\|g_{\infty}\|_{C^k}
$$
for a suitable constants $\tilde C_k$ and $\lambda_k$. It follows that $f(s)$ converges to $0$ in $C^{\infty}$-norm, i.e. for every positive integer $k$ there exists suitable constants $B_k$ and $\mu_k$ such that 
$$
\|f(s)\|_{C^k}\leq B_k\, e^{-\mu_kt}\,.
$$  
Therefore $\beta(t)$ converges in $C^{\infty}$-norm to
$$
\beta_{\infty}:=\int_0^{\infty}f(s)\,ds+\Omega_o^{2,0}\,.
$$ 
and \eqref{eq:HSflow} has a unique solution $\Omega(t)$ for $t\in[0,\infty)$ converging to 
$$
\Omega_{\infty}:=\omega_{\infty}+\beta_{\infty}+\bar\beta_{\infty}\,.
$$
Finally,  since $\Omega(t)$ is closed for every $t$, its limit is a symplectic form, as required. 

\begin{oss}
Generically we do not expect that $\beta(t)$ converges to zero. A trivial counterexample is the following:\\
consider the standard complex torus $\mathbb{T}^{2n}=\C^n/\R^{2n}$ with the standard flat K\"ahler structure $\omega_0=-i\sum dz^r\wedge d\bar z^r$. Then $\Omega_0=\omega_0+dz^1\wedge dz^2+d\bar z^1\wedge d\bar z^2$ is a Hermitian-symplectic
structure and
$
\Omega(t)\equiv \Omega_0
$
solves the flow \eqref{eq:HSflow}.
\end{oss}

\end{proof}

\subsection{Flow  \eqref{eq:HSflow} on Lie algebras}
Let $(\g,\mu)$ be a Lie algebra endowed with a complex structure $J$. Let $\{Z_r\}$ an arbitrary $(1,0)$-frame with dual frame $\{\zeta^k\}$.  Every Hermitian  inner product  $g$ on $(\g, \mu, J)$ can be written as
$$
g=g_{r\bar k}\,\zeta^{r} \zeta^{\bar k}\,,
$$
for some real coefficients $(g_{r\bar k})$. The  inner product  $g$ induces the fundamental form
$$
\omega=-i\,g_{r\bar k}\,\zeta^{r}\wedge\zeta^{\bar k}\,.
$$
Therefore an arbitrary non-degenerate $2$-form $\Omega$ dominating $J$ cane written as 
$$
\Omega=-i\,g_{r\bar k}\,\zeta^{r}\wedge\zeta^{\bar k}+\beta_{ij}\, \zeta^{i}\wedge \zeta^{j}+\bar \beta_{ij}\,  \zeta^{\bar i}\wedge  \zeta^{\bar j}
$$
Using equations \eqref{unitary}, we get that the problem \eqref{eq:HSflow} is equivalent to the following system
\begin{equation}\label{eqalgebras}
\begin{cases}
\frac{d}{dt}g_{i\bar j}=- \mu_{i\bar j }^a \mu_{ar}^{r}+\mu_{i\bar j }^ag^{r\bk}\mu_{r\bk}^{\bar l} g_{a\bar l}+
                 \mu_{i\bar j }^{\bar b}\mu_{\bar b \bar r}^{\bar r}-\mu_{i\bar j }^{\bar b}g^{k \br}\mu_{k \bar r }^{ l} g_{l\bar b}\\

\frac{d}{dt}\beta_{ij}=-i \mu_{ij }^a \mu_{ar}^{r}+i\mu_{ij }^ag^{r\bk}\mu_{r\bk}^{\bar l} g_{a\bar l}
                 +i\mu_{i j }^{\bar b}\mu_{\bar b \bar r}^{\bar r}-i\mu_{i j }^{\bar b}g^{k \br}\mu_{k \bar r }^{ l} g_{l\bar b}\\

                 g_{i\bar j}(0)=h_{i\bar j}\\
                 \beta_{i j}(0)=h_{i j}
\end{cases}
\end{equation}
where
$$
\Omega_0=-ih_{i\bar j}\zeta^{i}\wedge \zeta^{\bar j}+h_{rs}\zeta^{r}\wedge \zeta^{\bar s}+\bar{h}_{lm}\zeta^{\bar l}\wedge \zeta^{\bar m}
$$
is the starting  symplectic form  taming $J$ and
$$
\Omega=-ig_{i\bar j}\zeta^{i}\wedge \zeta^{\bar j}+\beta_{rs}\zeta^{r}\wedge \zeta^{\bar s}+\bar \beta_{lm}\zeta^{\bar l}\wedge \zeta^{\bar m}
$$
is the solution to \eqref{eq:HSflow}.

In real dimension four, the  equations \eqref{eqalgebras} can be simplified by writing
every $J$-Hermitian inner product  on $\g$ in matrix notation  as
$$
(g_{i\bar j})=\left(\begin{array}{cc}
x &z\\
\bar z & y\\
\end{array}
\right)
$$
where $x,y$ are positive real numbers and $z\in \C$ satisfying
$$
xy-|z|^2>0\,.
$$
In this way the inverse of $g$ is
$$
\left(g^{\bar ji}\right)=\frac{1}{xy-|z|^2}\left(\begin{array}{cc}
y &-z\\
-\bar z & x\\
\end{array}
\right)\,.
$$
\begin{example} Consider the solvable Lie algebra $\g$ with structure equations
$$
(24,-14,0,0)
$$
endowed with the complex structure
$$
J(e_1)=e_2\,,\quad J(e_3)=e_4\,.
$$
Let $\{Z_1,Z_2\}$ be the $(1,0)$-frame
$$
Z_1=\frac12 (e_1-ie_2)\,,\quad  Z_2=\frac12 (e_3-ie_4)\,;
$$
then
$$
[Z_1, Z_{\bar 1}]=[Z_2,Z_{\bar 2}]=0\,,\quad
[Z_1,Z_{\bar 2}]=-\frac{1}{2}\,Z_{1}\,,\quad [Z_{\bar 1},Z_{2}]=-\frac12\,Z_{\ov 1}\,,\quad
[Z_1,Z_2]=\frac12 Z_1\,.
$$

Consider the $(1,0)$-coframe
$$
\zeta^1=e^1+ie^2\,,\quad \zeta^2=e^3+i e^4
$$
dual to $\{Z_1,Z_2\}$. Then
$$
d\zeta^1=-i\,\zeta^1\wedge e^4\,,\quad d\zeta^2=0\,.
$$
There follows
$$
d\zeta^1=-\frac12\,\zeta^{12}+\frac12\zeta^{1\bar 2}\,,
$$
i.e.
$$
\partial \zeta^1=- \frac12\,\zeta^{12}\,,\quad \bar \partial \zeta^1= \frac{1}{2}\,\zeta^{1\bar 2}\,.
$$
The generic  $2$-form taming  the complex structure  $J$ is
$$
\tilde\Omega= -i x^2\, \zeta^{1\bar 1}- i y^2\, \zeta^{2\bar 2}-iz\, \zeta^{1\bar 2}-i\bar z \, \zeta^{ 2\bar 1}+w \, \zeta^{1 2}+\bar w\, \zeta^{\bar 1\bar 2}
$$
where $x,y\in \R$ and $z\,,w\in \C$ satisfy
$$
x^2y^2-|z|^2>0\,,
$$
Moreover, the closure of $\tilde \Omega$ implies $iz=w$, i.e. 
$$
\tilde\Omega= -i x^2\, \zeta^{1\bar 1}- i y^2\, \zeta^{2\bar 2}-iz\, \zeta^{1\bar 2}-i\bar z \, \zeta^{ 2\bar 1}+iz \, \zeta^{1 2}-i\bar z\, \zeta^{\bar 1\bar 2}
$$ 
The  Bismut Ricci form  with respect to $\tilde \omega:=\tilde \Omega^{1,1}$ is then given by
$$
\rho^B(\tilde{\omega})=i\frac{zx^2}{4(x^2y^2-|z|^2)}\,\zeta^{12}-i\frac{zx^2}{4(x^2y^2-|z|^2)}\,\zeta^{1\bar 2}-i\frac{\bar zx^2}{4(x^2y^2-|z|^2)}\,\zeta^{\bar 1\bar 2}+i\frac{{\bar z} x^2}{4(x^2y^2-|z|^2)}\,\zeta^{\bar 1 2}
$$
The HS flow reduces  to
$$
\begin{cases}
\dot x = \dot y =0,\\
\dot z=-\frac{zx^2}{4(x^2y^2-|z|^2)}\,,
\end{cases}
$$
with initial conditions $x(0) = x_o, y(0) =y_o, z(0) = z_o$. In particular $x$ and $y$ have to be constant and 
our system reduces to   
$$
\dot z=-\frac{zx_o^2}{4(x_o^2y_o^2-|z|^2)}\,,\quad z(0):=z_o\,.
$$
This last equation is radial in the sense that
its solutions $z=\rho\,e^{i\theta}$ have $\theta$ constant and our problem reduces to  
\begin{equation}\label{rho}
\dot \rho=-\frac{\rho x_o^2}{4(x_o^2y_o^2-\rho^2)}\,,\quad \rho(0)=\rho_o
\end{equation}
in terms of an unknown real function $\rho$. If $\rho_o$ is vanishing, then \eqref{rho} has solution $\rho\equiv 0$; otherwise its solution $\rho$ is defined and strictly positive in  $[0,\infty)$ and satisfies 
$$
\frac{\rho^2}{2x_o^2}-y^2_o\log (\rho)-\frac{\rho_o}{x_o^2}+\frac{y_o^2}{\rho_o^2}=\frac{t}{4}\,.
$$
This last relation ensures that $\rho$ tends to zero when $t$ tends to infinity. Therefore we have
$$
\begin{cases}
z_{\infty}=0\\
w_{\infty}=0\,,
\end{cases}
$$
and thus
$$
\Omega_{\infty}=-i x_o^2\, \zeta^{1\bar 1}-  i y_o^2\, \zeta^{2\bar 2}.
$$
\end{example}


\begin{thebibliography}{12}%%%%%%%%%%%%%%% BIBLIOGRAFIA%%%%%%%%%%%%%%%%%%%%%%%

\bibitem{AG} V. Apostolov, M. Gualtieri,  Generalized K\"ahler manifolds, commuting complex structures, and split tangent bundles,
{\em Comm. Math. Phys.}  {\bf 271} (2007), no. 2, 561--575.

\bibitem{Bismuth}
 J. M. Bismut, A local index theorem for non-K\"ahler manifolds, {\em Math.  Ann. } {\bf 284} (1989), no. 4, 681--699.


\bibitem{Eberlein}  P. Eberlein, Geometry of 2-step nilpotent Lie groups,  {\em Annales Scientiques de l'E.N.S.}  (1994), 611-660]

\bibitem{Enrietti} N. Enrietti, Static SKT metrics on Lie groups,   {\em Manuscripta Math.} {\bf 140} (2013), 557--571.
\bibitem{EF} N. Enrietti, A. Fino, Special Hermitian metrics and Lie groups,  {\em Differential Geom. Appl.} {\bf 29} (2011), suppl. 1, 211--219.

\bibitem{EFV}
N. Enrietti, A. Fino and  L. Vezzoni, Hermitian Symplectic structures and SKT metrics, {\em  J. Symplectic Geom.} {\bf10}, n. 2 (2012), 203--223.




\bibitem{FPS} A. Fino, M. Parton and  S. Salamon, Families of strong KT structures in six dimensions, {\em Comment. Math. Helv.} {\bf 79} (2004), no. 2, 317--340.

\bibitem{FT3} A. Fino, A. Tomassini, Non-K\" ahler solvmanifolds with generalized K\"ahler structure,  {\em  J. Symplectic Geom.}  {\bf 7}  (2009), no. 2, 1--14.

\bibitem{Gauduchon}
P. Gauduchon, Hermitian connections and Dirac operators, {\em Boll. Un. Mat. Ital. B (7)} {\bf 11} (1997), no. 2, suppl.,
257--288.

\bibitem{gauduchon2} P. Gauduchon, La 1-forme de torsione d'une vari\'et\'e hermitienne compacte, {\em Math. Ann.}
{\bf 267} (1984), 495--518.

\bibitem{Inoue} M. Inoue, On surfaces of class VII0, {\em  Invent. Math.}
{\bf  24} (1974), 269--310.

\bibitem{ivanov2}  S.  Ivanov,  G.  Papadopoulos, Vanishing theorems and string backgrounds, {\em Classical Quantum Gravity} {\bf 18} (2001), no. 6, 1089--1110.

\bibitem{lauret2}
J. Lauret, The Ricci flow for simply connected nilmanifolds. {\em Comm. Anal. Geom.} 1{\bf 9} (2011), no. 5, 831--854.

\bibitem{MS} T. Madsen and  A. Swann,  Invariant strong KT geometry on four-dimensional solvable Lie groups,  {\em J. Lie Theory}  {\bf 21} (2011), no. 1, 55--70.

\bibitem{LZ} T.-J. Li and  W. Zhang, Comparing tamed and compatible symplectic cones and
cohomological properties of almost complex manifolds, {\em Comm. Anal. Geom.}, {\bf 17} (2009), no. 4, 651--683.

\bibitem{Popovici} D. Popovici, Limits of Projective Manifolds Under
Holomorphic Deformations, preprint arXiv:1003.3605.


\bibitem{RT}  F. A. Rossi and  A. Tomassini, On strong K\" ahler and astheno-K\"ahler metrics on nilmanifolds, {\em Adv. Geom.}  {\bf 12} (2012), no. 3, 431--446.

\bibitem{ST}
J. Streets and  G. Tian, A parabolic flow of pluriclosed metrics, \emph{Int. Math. Res. Not. IMRN} 2010, no. 16, 3101--3133.

\bibitem{ST2}
J. Streets and  G. Tian,  Hermitian curvature flow, \emph{J. Eur. Math. Soc.} (JEMS) {\bf 13} (2011), no. 3, 601--634.

\bibitem{Vezzoni} L. Vezzoni, A note on Canonical Ricci forms on 2-step nilmanifolds,  {\em  Proc. Amer. Math. Soc.} {\bf 141} (2013), no. 1, 325--333. 

\end{thebibliography}
\end{document}